\documentclass{article}
\usepackage {amssymb}
\usepackage {amsfonts}
\usepackage {amsmath}
\usepackage {amsthm}
\usepackage{graphicx}
\usepackage{graphics}
\usepackage {framed}
\usepackage {amsxtra}
\usepackage {enumerate}
\usepackage[margin=1in]{geometry}

\newcommand{\restrict}{\mbox{$\mid\hspace{-1.1mm}\grave{}$}}

\theoremstyle{plain}
\newtheorem{thm}{Theorem}[section]
\newtheorem{cor}[thm]{Corollary} 
\newtheorem{lem}[thm]{Lemma} 

\newtheorem{prop}[thm]{Proposition}

\theoremstyle{definition}
\newtheorem{question}[thm]{Question}
\newtheorem{defn}[thm]{Definition}

\newtheorem*{Lem}{Lemma}
\newtheorem*{Exmp}{Example}

\theoremstyle{remark}

\begin{document}

\title{Sums of finitely many distinct rationals}

\author {David Hobby\footnote{hobbyd@newpaltz.edu\ \ State University of New York\ \ New Paltz, NY 12561} \and 
Donald Silberger\footnote{silbergd@newpaltz.edu\ \ State University of New York\ \ New Paltz, NY 12561} \and 
Sylvia Silberger\footnote{sylvia.silberger@hofstra.edu\ \ Hofstra University\ \ Hempstead, NY 11549}}

\maketitle

\centerline{\sf Remembering Jacqueline Bare Grace \  1942-2016}

\begin{abstract} ${\cal E}$ denotes the family of all finite nonempty $S\subseteq{\mathbb N}:=\{1,2,\ldots\}$, and ${\cal E}(X):={\cal E}\cap\{S:S\subseteq X\}$ when $X\subseteq{\mathbb N}$. 
Similarly, ${\cal F}$ denotes the family of all finite nonempty $T\subseteq{\mathbb Q}^+$, and ${\cal F}(Y) := {\cal F}\cap\{T:T\subseteq Y\}$ where ${\mathbb Q}^+$ is the set of all positive 
rationals and $Y\subseteq{\mathbb Q}^+$.
   
This paper treats the functions $\sigma:{\cal E}\rightarrow{\mathbb Q}^+$ given by $\sigma:S\mapsto\sigma S :=\sum\{1/x:x\in S\}$, the function $\delta:{\cal E}\rightarrow{\mathbb N}$ 
defined by $\sigma S = \nu S/\delta S$ where the integers $\nu S$ and $\delta S$ are coprime, and the more general function $\Sigma:{\cal F}\rightarrow{\mathbb Q}^+$ where $\Sigma T$  
denotes the sum of the elements in $T$ for $T\in{\cal F}$. 

{\bf Theorem \ref{Main1}.} \ For each $r\in{\mathbb Q}^+$, there exists an infinite pairwise disjoint subfamily ${\cal H}_r\subseteq{\cal E}$ such that $r=\sigma S$ for all $S\in{\cal H}_r$. 

{\bf Theorem \ref{Main2}.} \ Let $X$ be a pairwise coprime set of positive integers. Then $\sigma\restrict{\cal E}(X)$ and $\delta\restrict{\cal E}(X)$ are injective. Also, $\sigma C\in{\mathbb N}$ 
for $C\in{\cal E}(X)$ only if $C=\{1\}$.  

{\bf Theorem \ref{NotHyper}.} There is a set $X$ of positive rational numbers for which $\Sigma:{\cal F}(X)\rightarrow{\mathbb Q}^+$ is a surjection, but for which $1\in X$ and 
 the only $S\in{\cal F}(X)$ with $\Sigma S = 1$ is $S = \{1\}$. \end{abstract}

\section{Introduction}

The first portion of this paper deals with the reciprocals\footnote{These positive rational numbers have long been known also as ``Egyptian fractions''.} of the positive integers, and the shorter 
concluding portion deals analogously with the more general positive rationals. 

${\cal E}$ denotes the family of all nonempty finite subsets $S \subseteq {\mathbb N} := \{1,2,3,\ldots\}$, $\cal{F}$ denotes the family of all finite nonempty $T\subseteq \mathbb Q$, and ${\cal I}$ denotes its subfamily of finite intervals 
$[m,n] := \{m,m+1,\ldots,n-1,n\}$ of consecutive integers. The set of positive rational numbers is written ${\mathbb Q}^+$. The first portion of our paper is devoted principally to the function 
$\sigma:{\cal E}\rightarrow{\mathbb Q}^+$ defined by 
\[
\sigma: S\mapsto \sigma S\quad\mbox{where}\quad \sigma S := \sum_{x\in S}\frac{1}{x}.
\] 

For $r\in{\mathbb Q}^+$, the expression ${\cal E}_r$ denotes the family of all finite $S\subseteq{\mathbb N}$ for which $r=\sigma S$. 

\begin{thm}\label{Main1} For each $r\in{\mathbb Q}^+$, there exists an infinite pairwise disjoint subfamily ${\cal H}_r\subseteq{\cal E}_r$. \end{thm} 

The function $\sigma$ induces two other functions, $\nu:{\cal E}\rightarrow{\mathbb N}$ and $\delta:{\cal E}\rightarrow{\mathbb N}$, via the fact that for each $S\in{\cal E}$ there is a 
unique coprime pair $\langle \nu S,\delta S\rangle$ of positive integers for which $\sigma S = \nu S/\delta S$. We discuss both $\delta$ and $\sigma$.

When $X\subseteq{\mathbb N}$ then ${\cal E}(X) := {\cal E}\cap\{S:S\subseteq X\}$. Thus, e.g., ${\cal E}({\mathbb N}) = {\cal E}$.  

\begin{thm}\label{Main2} Let $X$ be a pairwise coprime subset of ${\mathbb N}$. Then $\sigma\restrict{\cal E}(X)$ and $\delta\restrict{\cal E}(X)$ are injections. Also, 
$\sigma C\in{\mathbb N}$ for $C\in{\cal E}(X)$ only if $C=\{1\}$. \end{thm}

Our work grew from our interest in the set $\sigma[{\cal I}]$ of ``harmonic rationals'', by which people mean the numbers that occur as sums of finite segments of the harmonic series 
\[1+\frac{1}{2}+\frac{1}{3}+\cdots+\frac{1}{j-1}+\frac{1}{j}+\frac{1}{j+1}+\cdots = \infty.\] It is well known and easy to see that $\sigma[{\cal I}]$ is dense in ${\mathbb R}^+$, but 
$\sigma[{\cal I}]\not={\mathbb Q}^+$ is true as well. Indeed,  L. Theisinger \cite{Theisinger} proved in 1915 that $\sigma[1,n] \in {\mathbb N}$ only if $n=1$, and in 1918 
J. K\"ursch\'ak \cite{Kurschak} proved that $\sigma[m,n] \in {\mathbb N}$ only if $m=n=1$. The latter fact is recalled, for instance, as Exercise 3 on Page 7 of \cite{Baker}. 

Other natural subfamilies of $S\in{\cal E}$ for which $\sigma S \not\in {\mathbb N}$ were noted later.  P. Erd\"os \cite{Erdos2}, also Page 157 of \cite{Hoffman}, extended the 
Theisinger-K\"ursch\'ak theorem to the finite segments of an arithmetic series:

\centerline{If \  $d \ge 1$, \ and if either \ $m>1$ \ or \ $k>1$, \  then \ $\displaystyle \sum_{j=0}^{k-1}\frac{1}{m+dj} \not\in{\mathbb N}$. }  

\noindent Erd\"os' result was carried further by H. Belbachir and A. Khelladi \cite{Belbachir}: 

\centerline{For $\{a_0,a_1,\ldots,a_{k-1}\}\subseteq{\mathbb N}$, if \ $d \ge 1$, and if either \ $m>1$ \ or \ $k>1$, \  then  \  
$\displaystyle \sum_{j=0}^{k-1}\frac{1}{(m+dj)^{a_j}} \not\in {\mathbb N}$. }

According to Erd\"os \cite{Erdos2}, looking beyond sums of distinct reciprocals R. Obl\'ath showed that 

\centerline{ $\displaystyle \sum_{i=m}^n \frac{a_i}{i} \not\in {\mathbb N}$ \ if \ $i$ \ is coprime to \ $a_i$ \ for each \ $i\in [m,n]$, \ 
where \ $[m,n] \not= \{1\}$. } \vspace{.3em}

\noindent We note that this sum of Obl\'ath fails to be an integer provided only that his $a_i$ are odd whenever $i$ is even.\vspace{.5em}

Every Theisinger-K\"ursch\'ak sort of result we mentioned specifies a case where $\sigma S \not\in {\mathbb N}$. Eventually we branched off into a side topic, which led to our rediscovering a result 
published \cite{Erdos3} in 1946:\vspace{.5em}

\noindent{\bf Theorem (Erd\"{o}s-Niven)} \ {\sl The function $\sigma\restrict{\cal I}$ is injective.}\vspace{.5em}

Our reinvention of this Erd\"{o}s-Niven wheel resulted in machinery that provoked us to consider an analogous surjectivity question; to wit: 
Is the range of $\sigma$ equal to ${\mathbb Q}^+$?  

Theorem \ref{Main1} answers this in the affirmative.

We prove Theorem \ref{Main1} in \S2 and Theorem \ref{Main2} in \S3. In \S4 we look again at a serendipitous gift.  In \S5 we initiate a study of the function $\Sigma:{\cal F}\rightarrow{\mathbb Q}^+$ 
where $\Sigma T$ denotes the sum of the $q\in T$ for $T\in{\cal F}$. 

\section{Surjectivity} 

The following equality holds for all complex numbers $z\notin\{-1,0\}$. Its utility earns it the name, {\em Vital Identity}:  
\[\frac{1}{z} = \frac{1}{z+1}+\frac{1}{z(z+1)},\] 
The Vital Identity serves as our main tool for proving Theorem \ref{Main1}, by giving us that $\sigma\{n\}=\sigma\{n+1,n(n+1)\}$ for all $n\in{\mathbb N}$. This fact can be usefully 
restated as \ $\sigma\{n\}=\sigma\{\diamond n,\star n\}$, where $\diamond:{\mathbb N}\rightarrow{\mathbb N}$ and $\star:{\mathbb N}\rightarrow{\mathbb N}$ are the strictly increasing functions 
defined by \ $\diamond:n\mapsto n+1$ \ and by \ $\star:n\mapsto n(n+1)$.

Each word $\mathbf{w}$ in the alphabet $\{\diamond,\star\}$ expresses a string of function compositions engendering a strictly increasing function $\mathbf{w}:{\mathbb N}\rightarrow{\mathbb N}$. Context will tell us when the word $\mathbf{w}$ is to be treated as an injection.\vspace{.5em}

An easy induction on $k\ge1$ shows that the integer $\star^kn$ has at least $k+1$ distinct prime factors if $n\ge2$. This seems less surprising when one contemplates that $\star^kn>n^{2^k}$.\vspace{.5em} 

Defining ${\cal E}_r := \{S:S\in{\cal E}\wedge \sigma S=r\}$, nobody will doubt that the 
family ${\cal E}/\sigma:=\{{\cal E}_r: r\in{\mathbb Q}^+\}$ is an infinite partition of the 
set ${\cal E}$. So, our 
only substantive task here is to show, for $r\in{\mathbb Q}^+$, that there is an infinite pairwise disjoint subfamily ${\cal H}_r\subseteq{\cal E}_r$, whence the family ${\cal E}_r$ itself is infinite.  

There are infinitely many pairs $\langle a,b\rangle\in{\mathbb N}^2$ for which $r=a/b$. For the sake of convenience, we will choose and fix $b$; we pick this $b\ge2$ in order to avoid unessential issues due 
to the fact that $\diamond1=\star1$. We then begin by constructing an infinite pairwise disjoint subfamily ${\cal G}_{1/b}\subseteq{\cal E}$ for which $1/b=\sigma S$ whenever $S\in{\cal G}_{1/b}$.\vspace{.5em} 

The expression $\mathbf{W}$ denotes the set of all finite words $\mathbf{w}$ in the letters $\diamond$ and $\star$.\footnote{Many 
semigroup theorists would write $\mathbf{W}$ as $\{\diamond,\star\}^*$. It is the free monoid on the 
letters $\diamond$ and $\star$.} The length of the word $\mathbf{w}$ is written $|\mathbf{w}|$. 
When the word $\mathbf{w}$ is interpreted as a function on ${\mathbb N}$, we show this with parentheses, writing $\mathbf{w}(n)$.  As a function, the empty word of length zero is the identity function on ${\mathbb N}$.

For $k\ge0$ we define $\mathbf{W}_k$ to be the set of all $\mathbf{w} \in \mathbf{W}$ with $|\mathbf{w}|=k$, and $\mathbf{W}_k (n)$ denotes the multiset of all integers of the form $\mathbf{w}(n)$ for $n \in \mathbb{N}$ and $\mathbf{w} \in \mathbf{W}_k$. Similarly, $\mathbf{W}(n)$ denotes 
the multiset of all $\mathbf{w}(n)$ for $n \in \mathbb{N}$ and $\mathbf{w} \in \mathbf{W}$. 

We will need to deal with the fact that $\mathbf{w}(n)=\mathbf{v}(n)$ can happen for some $n  \in \mathbb{N}$ while $\mathbf{w}$ and $\mathbf{v}$ are distinct elements of $\mathbf{W}$.\vspace{.5em}

For example, when $\mathbf{w} := \diamond\star^2\diamond^3$ then $\mathbf{w}(n) = \diamond\star^2\diamond^3n = \diamond\star^2(n+3) = \diamond\star((n+3)(n+4)) = 
\diamond((n+3)(n+4)((n+3)(n+4)+1)) =  (n+3)(n+4)((n+3)(n+4)+1)+1$. Thus $\mathbf{w}(1) = 421$. Incidentally, $|\mathbf{w}| = |\diamond\star^2\diamond^3| = 6$. 
So $\mathbf{w} \in \mathbf{W}_6$ and 
$421\in \mathbf{W}_6(1)$.  
Also, although $\diamond^4$ and $\star$ are distinct as words, $\diamond^4(2) = 6 = \star(2)$. However, notice that $|\diamond^4| = 4 \neq 1 = |\star|$.\vspace{.5em}

It is useful to write each word $\mathbf{w} \in \mathbf{W}$ in the format $\mathbf{w} = \diamond^{j_r}\star\diamond^{j_{r-1}}\star\cdots\diamond^{j_1}\star\diamond^{j_0}$ where $j_i\ge 0$, since the roles played 
by the basic components $\diamond$ and $\star$ in our story will differ. Of course then $|\mathbf{w}| = r+\sum_{i=0}^r j_i$.

Bearing in mind that $b \geq 2$, the following lemma is obvious.

\begin{lem}\label{Lem2.1} Let $b\le m<\star m\le n<\star(m+1)$. Then $\mathbf{w}(b)=n$ if and only if either $\mathbf{w}=\diamond^{n-b}$ or there is an integer $k\in[b,m]:=\{b,b+1,\ldots m\}$ and a possibly empty word ${\mathbf u}$ 
such that ${\mathbf u}(b) = k$ and such that $\mathbf{w}=\diamond^{n-\star k}\star{\mathbf u}$.  \end{lem}

\begin{lem}\label{Lem2.2} Let $n>b>1$. Then $\diamond^{n-b}b = n$. But if $\diamond^{n-b}\not=\mathbf{w}$ while $\mathbf{w}(b)=n$ then $|\mathbf{w}|<n-b$.  \end{lem}

\begin{proof}  Obviously $\diamond^{n-b}b=n$ and $|\diamond^{n-b}|=n-b$. It is also clear that if $\star$ is a letter in the word $\mathbf{w}$ then fewer than $n-b$ compositional steps are needed to reach $n$, since $\star(m) > m+1$ for all $m \geq 2$.  Thus $|\mathbf{w}|<n-b$.   \end{proof}

\begin{lem}\label{Lem2.3} Let $b\le k<k'\le m<\star(m)\le n<\star(m+1)$. Let $\mathbf{w}(b)=n={\mathbf w'}(b)$ where $\mathbf{w}=\diamond^{n-\star(k)}\star{\mathbf u}$ with ${\mathbf u}(b)=k$, and where 
${\mathbf w'}=\diamond^{n-\star(k')}\star{\mathbf u'}$ with ${\mathbf u'}(b)=k'$. Then $|{\mathbf w}|>|\mathbf{w'}|$.  \end{lem} 

\begin{proof} Since $|\mathbf{w}|=n-\star(k)+1+|{\mathbf u}|$ and $|{\mathbf w'}|= n-\star(k')+1+|{\mathbf u'}|$, it suffices to prove $-\star(k)+|{\mathbf u}|>-\star(k')+|{\mathbf u'}|$; i.e., that $\star (k')-\star(k) > |\mathbf{u'}|-|\mathbf{u}|$. 
But $k\le k'-1$, and $k'-b$ is by Lemma \ref{Lem2.2} the length of the longest word ${\mathbf v}$ for which ${\mathbf v}(b)=k'$. So $\star(k')-\star(k) = k'(k'+1)-k(k+1) \ge k'(k'+1)-(k'-1)k'=2k'>k'-b\ge|{\mathbf u'}|>|\mathbf{u'}|-|\mathbf{u}|$.  
\end{proof}

\begin{thm}\label{Th2.4} Let $\mathbf{w}(b)=n={\mathbf w'}(b)$ with $\mathbf{w} \not= \mathbf{w'}$. Then $|\mathbf{w}|\not=|\mathbf{w'}|$.  \end{thm}

\begin{proof} We will argue by induction on $n\ge b$. For $n=b$ the theorem is obvious. So pick $n>b$. Suppose for every $s \in [b,n-1]$ that, if $\mathbf{v} \not= \mathbf{v'}$ are words with ${\mathbf v}(b) = s = {\mathbf v'}(b)$, then 
$|\mathbf{v}|\not=|\mathbf{v'}|$. By Lemma \ref{Lem2.1} we can write $\mathbf{w}=\diamond^{n-\star(k)}\star{\mathbf u}$ and ${\mathbf w'}=\diamond^{n-\star(k')}\star{\mathbf u'}$ where ${\mathbf u}(b)=k$ and ${\mathbf u'}(b)=k'$. Without loss of generality 
$b\le k\le k'<n$. 

{\em Case:} $k=k'$. Then $|\mathbf{w'}|-|\mathbf{w}|=|\mathbf{u'}|-|\mathbf{u}|$. By the inductive hypothesis, $|\mathbf{u'}|=|\mathbf{u}|$ if and only if $\mathbf{u'}=\mathbf{u}$. But $\mathbf{u'}=\mathbf{u}$ if and only if $\mathbf{w'}=\mathbf{w}$ in the present Case. So $\mathbf{w'}=\mathbf{w}$ 
if and only if $|\mathbf{w'}|=|\mathbf{w}|$ in this Case. 

{\em Case:} $k<k'$. Then $|\mathbf{w}|>|\mathbf{w'}|$ by Lemma \ref{Lem2.3}. So in this Case too, if $\mathbf{w'} \not= \mathbf{w}$ then $|\mathbf{w'}|\not=|\mathbf{w}|$.     \end{proof} 

\begin{cor}\label{Cor2.5} For every integer $k\ge0$, the multiset ${\mathbf W}_k(b)$ is a simple set. So $|{\mathbf W}_k(b)| = |{\mathbf W}_k|=2^k$. \end{cor}

\begin{lem}\label{Lem2.6} If $k$ is a nonnegative integer then $\sigma({\mathbf W}_k(b))=1/b$.  \end{lem} 

\begin{proof} We induce on $k\ge0$. {\em Basis Step:} $\sigma({\mathbf W}_0(b)) = \sigma\{b\} = 1/b$. {\em Inductive Step:} Pick $k\ge0$, and suppose that $\sigma({\mathbf W}_k(b))=1/b$. The family ${\cal S} := 
\{\{\diamond{\mathbf v}(b),\star{\mathbf v}(b)\}:{\mathbf v} \in {\mathbf W}_k(b) \}$ 
is pairwise disjoint by Corollary \ref{Cor2.5}, and $\bigcup{\cal S}={\mathbf W}_{k+1}(b)$. But the Vital Identity implies that $\sigma\{\diamond{\mathbf v}b,\star{\mathbf v}b\} = \sigma\{{\mathbf v}b\}$ 
for every ${\mathbf v \in W}$. Hence $\sigma({\mathbf W}_{k+1}(b))=\sigma({\mathbf W}_k(b))=1/b$.   \end{proof}

\begin{cor}\label{Cor2.7} Let $b\ge2$. Then there exists an infinite pairwise disjoint family ${\cal G}_{1/b}\subseteq{\cal E}$ such that $\sigma S = 1/b$ for every $S \in {\cal G}_{1/b}$.  \end{cor}

\begin{proof} Let $k_1:=0$. Pick $j\ge0$, and suppose for each integer $i \in [0,j]$ that the integer $k_i$ has been chosen so that $k_1<k_2<\cdots<k_j$ and such that the family 
$\{{\mathbf W}_{k_1}(b),{\mathbf W}_{k_2}(b),\ldots, {\mathbf W}_{k_j}(b)\}\subseteq{\cal E}$ is pairwise disjoint. Notice for each integer $t\ge2$ that $\diamond^t(b)=\min{\mathbf W}_t(b)<\max{\mathbf W}_t(b)=\star^t(b)$. 
So ${\mathbf W}_t(b)\cap\bigcup\{{\mathbf W}_{k_i}(b) :i \in [1,j]\}=\emptyset$ if $t>\star^{k_j}(b)$. We therefore can define $k_{j+1} := 1+\star^{k_j}(b)$ with the assurance that then the family 
$\{{\mathbf W}_{k_i}(b): i \in [1,j+1]\}\subseteq{\cal E}$ is pairwise disjoint. Let ${\cal G}_{1/b} := \{{\mathbf W}_{k_i}(b):i \in {\mathbb N}\}$. The corollary follows by Corollary \ref{Cor2.5} and 
Lemma \ref{Lem2.6}.    \end{proof}

It is now easy to finish establishing Theorem \ref{Main1}:

\begin{proof} We diminish clutter by writing $B_i := {\mathbf W}_{k_i}(b)$ for the ${\mathbf W}_{k_i}(b)$ in our proof of Corollary \ref{Cor2.7}, and letting ${\cal G}_{1/b} := \{B_1,B_2,\ldots\}$ be as 
promised by Corollary \ref{Cor2.7}. Recalling that $r=a/b$ and that $\sigma B_i=1/b$ for every $i$, we partition the set ${\cal G}_{1/b}$ into a family of $a$-membered subsets; e.g., 
this family
 could be $\{{\cal C}_1,{\cal C}_2,\ldots\}$ where ${\cal C}_1 :=\{B_1,B_2,\ldots, B_a\},\, {\cal C}_2:=\{B_{a+1},B_{a+2}\ldots, B_{2a}\},\, {\cal C}_3:=\{B_{2a+1},B_{2a+2},\ldots, B_{3a}\},\ldots$  Let 
$D_k := \bigcup{\cal C}_k$ for each $k \in {\mathbb N}$, and define an infinite pairwise disjoint subfamily ${\cal H}_{a/b} :=\{D_1,D_2,\ldots\}\subseteq{\cal E}$. Since obviously $\sigma D_i = a/b = r$ for every 
$i \in {\mathbb N}$, Theorem \ref{Main1} is established. \end{proof}

Reviewing the argument above, we notice that there are infinitely many ways to partition the set ${\cal G} _{1/b} $ into a family of $a$-membered subsets, and thus to obtain alternative families of 
$a$-membered sets whose unions comprise the membership of other candidates to the title ${\cal H}_{a/b}$ besides the family to which we have given that name. That is to say, the 
Vital Identity confers on us an ability to produce infinitely many subfamilies of ${\cal E}$, any one of which could legitimately be called ${\cal H}_{a/b}$. Of course all of these candidates 
are subfamilies of ${\cal E}_r := \{S:r=\sigma S\}\cap{\cal E}$ --- given, as we are, that $r=a/b$. Moreover, there are further encumberances to a specification of every possible ${\cal E}_r$.  
We now glance at a few of them.\vspace{.5em}

First, other 
algorithms may yield members of ${\cal E}_r$ which the Vital Identity cannot provide.  One example is the greedy algorithm that keeps subtracting the largest possible element of $\{ 1/n : n \in \mathbb{N}
\}$ from $r$ until nothing remains.
\vspace{.5em}

Second, all of the families ${\cal E}_{a/b}$ remarked in the preceding paragraph utilized only the expression of the rational $r$ as its fractional form, $a/b$ for a specific pair $\langle a,b\rangle$ 
with $b\ge2$. But $r = a'/b'$ for infinitely many $\langle a',b'\rangle \in {\mathbb N}\times{\mathbb N}$. Each of these $\langle a',b'\rangle$ provides additional families of finite sets 
$S\subseteq{\mathbb N}$ for which $r=\sigma S$.\vspace{.5em}

Third, there are other procedures, besides the one elaborated in the propositions proved above, whereby for $s \in {\mathbb Q}^+$ the Vital Identity uncovers infinite subfamilies of ${\cal E}_s$. One such 
of these alternative procedures involves the generation -- from an arbitrary ``seed'' $A_1 \in {\cal E}$ -- of an infinite sequence $\langle A_i\rangle_{i=1}^\infty$ of 
sets whose terms we take pains to 
make simple. Indeed, we arrange for $\langle A_i\rangle_{i=1}^\infty$ to be a sequence in ${\cal E}_{\sigma(A_1)}$. 

It is our guess that each such sequence has an infinite subsequence, the family of whose terms is pairwise disjoint. If our guess gets verified, then a duplication of the final portion of our proof above of 
Theorem \ref{Main1} will provide another route to that theorem, via a different class of pairwise disjoint subfamilies of ${\cal E}_s$. 

Anyway, these sequences of sets are sufficiently interesting to justify our briefly laying the groundwork for their future study. Moreover, they do give us new infinite subfamilies of ${\cal E}_s$.\vspace{.3em}

Call an integer $r_i$ {\em replaceable for} $A_i$ iff $r_i$ is the least element $x \in A_i$ such that $\{\diamond x,\star x\}\cap A_i=\emptyset$. Plainly each $A_i$ contains exactly one replaceable element. 
The recursion that generates $\langle A_i\rangle_{i=1}^\infty$ is given by $A_{i+1} := (A_i\setminus\{r_i\})\cup\{\diamond r_i,\star r_i\}$. Of course $|A_{i+1}|=|A_i|+1$. By the Vital Identity, 
$\sigma A_i=\sigma A_1$ for all $i \in {\mathbb N}$. We refer to $\langle A_i\rangle_{i=1}^\infty$ as the {\em $\sigma$-sequence from seed $A_1$}. It is obvious that each such sequence is  infinite. 

\begin{Exmp} To identify the replaceable element $r_i$ for $A_i := \{3,4,5,10,12,30\}$ we work upward from $\min A_i$.  We see that $r_i\not=3$ because $\{\diamond3,\star3\}=
\{4,12\}\subseteq A_i$, and $r_i\not=4$ because $\diamond4=5 \in A_i$, and $r_i\not=5$ because $\star5=30 \in A_i$. So $r_i=10$, since $\{\diamond10,\star10\} = \{11,110\}$ while $\{11,110\}\cap A_i=\emptyset$. 
The fact that $\{\diamond12,\star12\}=\{13,156\}$ while $\{13,156\}\cap A_i=\emptyset$ nominates $12$ as a candidate for $r_i$; but $12$ loses the election to $10$ since $10<12$.  Candidate $30$ gets even fewer 
votes than $12$ got. So $A_{i+1} = \{3,4,5,11,12,30,110\}$.\end{Exmp}

If $r_i = \min A_i$, and if $r_i$ is replaceable, then we say that $r_i$ is {\em doomed} in $A_i$. Since our sequence-generating recursion never introduces into $A_{i+1}$ an integer smaller than $\min A_i$, the 
sequence $\langle \min A_i\rangle_{i=1}^\infty$ is nondecreasing. If $d$ is doomed in $A_i$ then $d = \min A_i<\min A_{i+1}$, and $d\notin A_j$ for all $j>i$. If $\min A_i$ is not doomed in $A_i$ then surely 
$\min A_{i+1}=\min A_i$. Clearly, our guess above is equivalent to our surmise that $\lim_{i\rightarrow\infty}\min A_i = \infty$.\vspace{.3em}

En route to our guess that every $\sigma$-sequence $\langle A_j\rangle_{j=1}^\infty$ has an infinite pairwise disjoint subsequence $\langle A_{j_i}\rangle_{i=1}^\infty$, we experimented with the recursion 
operating from several different seed sets. We report on the nondefinitive results we got with the seed $A_1 := \{2\}$. The first six terms of this sequence are: $A_{j_1} = A_1 = \{2\};\, 
A_{j_2}=A_2=\{3,6\};\,A_3=\{4,6,12\};\, A_4=\{5,6,12,20\};\,A_{j_3}=A_5=\{5,7,12,20,42\};\,A_6=\{6,7,12,20,30,42\}$. In order to reach a secure $A_{j_4}$ we must have that 
$\min A_{j_4}>\max(A_{j_1}\cup A_{j_2}\cup A_{j_3})=\max\{2,3,5,6,7,12,20,42\}$. Since the first term of $\langle A_j\rangle$ with $\min A_j=7$ is $A_{27}$, we expect hours of 
pen work before $A_{j_4}$ is reached.  A kid with a computer would help. But, even Hal may drag its electronic feet before giving us, say\ldots \ $A_{j_{10^6}}$.\vspace{.3em} 

How fast does the integer sequence $\langle j_i\rangle=\langle 1,2,5,\ldots\rangle$ increase? We believe the sequence is infinite. Is it?\vspace{.5em}

\noindent{\bf Problem.} Considered in these three lights, an exhaustive treatment of the full family ${\cal E}_r$ remains at issue. One would like to recognize all of the $S \in {\cal E}$ 
for which it must happen that $r=\sigma S$. We do not have this information even in the restricted case that $r \in {\mathbb N}$. Indeed, it would be germane to know this for $r=1$. \vspace{.5em} 

\noindent{\bf Conjecture.} When $r=a/b \in {\mathbb Q}^+$ with $a$ and $b\ge1$ coprime, and when $c \in {\mathbb N}$, then there is a partition ${\cal U}_{\langle r,c\rangle}\subseteq{\cal E}$ 
of ${\mathbb N}\cap[cb,\infty)$ such that $\sigma S = r$ for every $S\in {\cal U}_{\langle r,c\rangle}$. This would strengthen Theorem \ref{Main1}.

\section{Injectivity revisited}

Lower case Greek letters always denote functions of a set variable, except where those symbols may be highjacked to designate numerical values assumed by such functions. For instance, the numbers  
$\sigma X$ and $\sigma X'$ may sometimes be abbreviated to $\sigma$ and $\sigma'$, respectively, when context obviates ambiguity.

Recall that the function  $\sigma:{\cal E}\rightarrow{\Bbb Q}^+$ induces two other functions, $\nu:{\cal E}\rightarrow{\Bbb N}$ and $\delta:{\cal E}\rightarrow{\Bbb N}$, via the fact that 
$\sigma X = \nu X/\delta X$ for a unique coprime pair $\nu X$ and $\delta X$ of positive integers. 

The least common multiple $\mu X$ of the integers in $X$ is useful for our project, since $\mu X/x$ is an integer for each $x \in X$, and so $\sigma X\cdot \mu X$ is an integer. Thus the equality 
$\sigma = \sigma\mu/\mu$ provides an easy presentation of $\sigma X$ as a fraction of integers.  Of course the lowest terms reduction of the fraction $\sigma\mu/\mu$ is $\nu/\delta$.\vspace{.5em}
   
We write $m|n$ to state that $m$ divides $n$. For $\{m,n\}\subseteq{\mathbb N}$ and $v\ge0$, the expression $m^v\|n$ is read ``$\,m^v$ {\em exactly divides} $n$'', and means that both 
$m^v|n$ and $m^{v+1} \not| \; n$.\vspace{.5em} 

We evoke two classic results, both of which are proved in \cite{Erdos2}. The first was conjectured by J. Bertrand but established by P. Chebyshev. The second, due to 
J. J. Sylvester \cite{Sylvester}, extends the first. 

The following fact is known either as Bertrand's Postulate or as Chebyshev's Theorem.   

\begin{thm} {\bf (Chebyshev)} \label{Chebyshev} If $n \ge 2$ then there is a prime $p$ such that $n < p < 2n$. \end{thm}

We also have the following.

\begin{thm} {\bf (Sylvester)} \label{Sylvester} 
If $k$ and $n$ are natural numbers with $k < n$, then there is a prime $p$ greater than $k$ that divides the product 
$n(n+1)(n+2) \dots (n+k-1)$.
\end{thm} 

These theorems give us the following corollary.

\begin{cor} \label{Sylvester corollary}
Let $1\le m<n$ be integers. Then there exists a power $p^v\ge2$ of a prime $p$ such that $p^v\|\mu[m,n]$ but also such that $p^v$ divides one and only one element $x\in[m,n]$. Moreover, $p^v\|x$. 
\end{cor}
  
\begin{proof} Let $m < n$ as above.

Assume $n \geq 2m$.  We let $k$ be the largest even number with $k \leq n$ and have $k \geq m$.  Then Theorem \ref{Chebyshev} gives us a prime $p$ with $m \leq k < p < 2k \leq n$, and we take $v = 1$, so $p^v | \mu[m,n]$. We have $k+1 \leq p$, so $2p \geq 2k + 2 > n$, showing that there is only one $x \in [m,n]$ with $p^v | x$.

If $n < 2m$, we take $k=n-m$ and use Theorem \ref{Sylvester} to get a prime $p  > k$ that divides $(m+1) (m+2) \dots (m + k) = (m+1)(m+2) \dots n$.  We again take $v = 1$, and have $p^v | \mu[m,n]$.  Since $p > k$, $p^v$ is the only $x \in [m,n]$ with $p^v | x$. \end{proof}

This corollary has legs:

\begin{defn}\label{sylvesdefine} For $X \in {\cal E}$, when $v \in {\Bbb N}$ and $p$ is a prime integer, we call $p^v$  a {\em sylvester power} for $X$ iff $p^v\|\mu X$ while $p^v|x$ for exactly one 
$x \in X$. The expression ${\mathbf S}(X)$ denotes the set of all sylvester powers for $X$. \end{defn} 

If $p^v\|\mu X$ while $p^v>\max X-\min X$, then surely $p^v \in {\mathbf S}(X)$. We proceed to set the stage.\vspace{.5em} 

\begin{Exmp} $[1000,1004] = \{1000,1001,1002,1003,1004\}=\{2^3\cdot5^3,7\cdot11\cdot13,2\cdot3\cdot167,17\cdot59,2^2\cdot251\}$.
Thus, the set of sylvester powers for this interval is ${\mathbf S}[1000,1004] = \{2^3,5^3,7,11,13,17,59,167,251\}$. The sylvester powers for an interval can be numerous.\end{Exmp} 

If $1<m<n$ then $|{\mathbf S}[m,n]|\ge2$, and indeed $2^v \in {\mathbf S}[m,n]$ for some $v \in {\mathbb N}$. The latter fact comes from 

\begin{lem}\label{2andmu} Let $1\le m<n$ be integers. Then there exists $2^v\in{\bf S}([m,n])$.  Moreover, $2^{v+1} > |[m,n]|$. \end{lem}

\begin{proof} Surely $2^v\|\mu[m,n]$ for some $v\ge 1$. So $2^v\|x$ for some $x\in[m,n]$, and $x=2^va$ for an odd integer $a$. But $2^va+2^v = 2^v(a+1)$ is the smallest multiple of $2^v$ greater than $x$. So $2^{v+1}|x+2^v$ since $a+1$ is even. Hence $x+2^v > n$, since otherwise $2^{v+1}|\mu[m,n]$ contrary to $2^v\|\mu[m,n]$. Similarly, $x-2^v < m$ since  $2^{v+1}|(x-2^v)$. Thus $2^v\in{\bf S}([m,n])$, and $2^{v+1} = x+2^v-(x-2^v) \ge n-m+2 > n-m+1  = |[m,n]|$. So $2^{v+1} > |[m,n]|$. \end{proof}

\begin{lem}\label{sylves1}   Let $X \in {\cal E}$, let $p$ be prime, let $v \in {\mathbb N}$, and let $p^v \in {\mathbf S}(X)$. Then $p^v\|\delta X$.\end{lem}

\begin{proof} Recall that $\sigma X= \sigma X\cdot\mu X/\mu X = \nu X/\delta X$ where $\nu X$ and $\delta X$ are coprime. Since $p^v$ is sylvester for $X$, there is a unique multiple $x$ of $p^v$ in $X$. 
Then $p|(\mu X/z)$ for all $z \in X\setminus\{x\}$, but  $\mu X/x$ is coprime to $p$. Therefore $\sigma X\cdot\mu X = \sum\{\mu X/t: t \in X\}$ is coprime to $p$. So $\nu X$ is coprime to $p$. The lemma 
follows.   \end{proof}

\begin{thm}\label{sylves4} For $\{X,Y\} \subseteq {\cal E}$ and $v \in {\mathbb N}$, let $p^v \in {\mathbf S}(X)\setminus{\mathbf S}(Y)$ with $p^v > \max Y - \min Y$. 
Then $\delta X \not= \delta Y$, and so $\sigma X \not= \sigma Y$. \end{thm}

\begin{proof} There exists $u\ge0$ with $p^u\|\mu Y$. If $u=v$ then the size of $p^v$ implies that $p^v \in {\mathbf S}(Y)$, contrary to hypothesis. So $u\not=v$. If $u>v$ then $p^u \in {\mathbf S}(Y)$. But 
$p^v \in {\mathbf S}(X)$ by hypothesis. So $\delta X\not=\delta Y$ by Lemma \ref{sylves1}. 

Now let $u<v$. Then $\neg(p^v\|\delta Y)$, since $p^u\|\mu Y$ implies that $p^t\|\delta Y$ only if $t\le u$. Again $\delta X \not=\delta Y$. \end{proof}

Lemmas \ref{2andmu} and \ref{sylves1} immediately establish the classic and already cited following result.

\begin{thm} {\bf (Theisinger-K\"ursch\'ak)} \label{kurschak} If $\sigma[m,n]$ is an integer then $m = n = 1$. \end{thm}  

The notion of a sylvester power suggests a way of strengthening the Erd\"{o}s-Niven theorem. The number of quadruples $m<n<m'<n'$ for which ${\mathbf S}[m,n]={\mathbf S}[m',n']$ seems to be finite. 
The only such quadruples of which we are aware are the two giving us ${\mathbf S}[4,7]=\{2^2,3,5,7\}={\mathbf S}[20,21]$ and ${\mathbf S}[5,7]=\{2,3,5,7\}={\mathbf S}[14,15]$. A consequence would be that 
almost always $\delta[m,n]\not=\delta[m',n']$ when $1<m<n<m'<n'$. 

We hope that a modification of a proof of Sylvester's Theorem could establish our\vspace{.5em} 

\noindent{\bf Conjecture.} If $1<m<n<m'<n'$ and if $n-m\le n'-m'$ then ${\mathbf S}[m,n]\not={\mathbf S}[m',n']$.\vspace{.5em}   

For each divergent subseries $1/{\mathbf x}:=\sum_{i=1}^\infty 1/x_i$ of the harmonic series, if ${\cal I}({\mathbf x})$ is the family of finite segments of ${\mathbf x}:=\langle x_i\rangle_{i=1}^\infty$, then 
$\{\sigma X: X \in {\cal I}({\mathbf x})\}$ is dense in ${\mathbb R}^+$. For which such ${\mathbf x}$ is $\sigma\restrict{\cal I}({\mathbf x})$ injective? 

The prime reciprocals series  $1/{\mathbf p}:=1/2+1/3+1/5+1/7+\cdots$ diverges. Let ${\mathbb P}:=\{p_1<p_2<p_3<\cdots\}$ be the set of all primes.  Are $\sigma\restrict{\cal E}({\mathbb P})$ and 
$\delta\restrict{\cal E}({\mathbb P})$ injective? In general, if $1/{\mathbf d} := \sum_{i=1}^\infty 1/d_i$ is a divergent subseries of the harmonic series, with $D:=\{d_i:i \in {\mathbb N}\}$ pairwise coprime, 
then must $\delta\restrict{\cal E}(D)$ and $\sigma\restrict{\cal E}(D)$ be injective? Our Theorem \ref{Main2} answers such questions affirmatively. So we now prove Theorem \ref{Main2}. 

\begin{proof} Let $A$ and $B$ be distinct nonempty finite subsets of the pairwise coprime set $X\subseteq{\mathbb N}$. Then without loss of generality there exist $a \in A\setminus B$ and a prime $p$ which 
divides $a$ but which is coprime to every $y \in (A\cup B)\setminus\{a\}$. Then $p|\delta A$ but $\neg(p|\delta B)$. Therefore $\delta A \not=\delta B$, and so $\sigma A\not= \sigma B$. As for the theorem's 
final claim, if $C$ is a finite nonempty subset of $X$ then $\sigma C\in {\mathbb N} \Leftrightarrow \delta C = 1 \Leftrightarrow C=\{1\} \Leftrightarrow \sigma C=1$.   \end{proof}

The function $\nu\restrict{\cal E}(X)$ is not injective: $\nu\{n\}=1, \nu\{3,13\}=16=\nu\{5,11\},\nu\{5,13\}=18=\nu\{7,11\},\ldots$

It is reasonable to ask what we can say about $\nu[{\cal E}(X)]$ in terms of $X$.  For example, for what nonsingleton sets $C\not=D$ in ${\cal E}(X)$, does $\nu C=\nu D$ hold? \vspace{.5em}

Like the functions $\star^k$, at which we shall glance in the next section, the function $\nu$ is potentially useful as a hunter of prime integers. For, if $q_1,\ldots,q_k$ are any $k$ distinct primes, and if 
$e_1,\ldots,e_k$ is any $k$-length sequence of positive integers then $\nu \{q_1^{e_1},q_2^{e_2},\ldots,q_k^{e_k}\}$ is coprime to each of these $q_i$.

\section{Stars}

Let ${\star}^\bullet b$ denote the sequence $\langle \star^j(b)\rangle_{j=1}^\infty$ of iterations of the function $\star:x\mapsto x(x+1)$ applied to a starting element $b\in{\mathbb N}$, and let ${\mathbb P}_b$ be the set of all primes which divide some term in the sequence $\star^\bullet b$. Since $\star^{j-1}(b)|\star^j(b)$ for all $j\in{\mathbb N}$, and since $x$ is coprime to $(\star(x))/x>1$ whenever $x\ge2$, we see that for each $j\ge1$ the integer $\star^j(b)$ is divisible by some prime that is coprime to $\star^i(b)$ for every $i\in[0,j-1]$. So the set ${\mathbb P}_b$ is infinite. It is easy also to see that, if $p^v\|\star^tb$ for some $t\ge 0$, then $p^{v+1}$ divides no term of $\star^\bullet(b)$.

In contrast to the limitations on $\star^{\bullet} b$ in the subsequent discussion, it is clear that for every finite set $S :=\{ p_1^{e(1)}, p_2^{e(2)},\ldots, p_n^{e(n)} \}$  of powers of distinct primes, there is a star sequence  $\star^\bullet b$  such that  $p_i^{e(i)} \| \star^j(b)$  for every  $i \in [1,n]$  and for every term  $\star^j(b)$  of the sequence  $\star^\bullet b$. Simply take $b$ to be the product of those prime powers.

\begin{thm}\label{star} Let $p \in {\mathbb P}_b$, let $i := i(b,p) \ge 0$ be the least integer with $p|\star^ib$, and hence with $p^n\|\star^ib$ for some $n := n(b,p) \ge 1$.  
Then $p^n\|\star^jb$ for every $j > i$.  
\end{thm}

\begin{proof} Let  $p^n\|\star^jb$. Since $\star^{j+1}b := \star^jb(\star^jb+1)$ and $\star^jb$ is coprime to $\star^jb+1$, we see that $p^n\|\star^{j+1}b$.  
\end{proof}

It is natural to ask: For which $b\in{\mathbb N}$, if any, does it happen that ${\mathbb P}_b = {\mathbb P}$, where ${\mathbb P}$ is the set of all primes?  

Let $p\in{\mathbb P}$. A term $\star^j(b)$ of sequence $\star^\bullet b$ is a multiple of $p$ if and only if $\star^j(b) \equiv_p0$; i.e., iff \ $\star^j(b)\equiv0({\rm mod}\,p)$. So we will work in the field ${\mathbb Z}_p$, and let 
$b\in{\mathbb Z}_p$. In this context, $p\in{\mathbb P}_b$ if and only if the sequence $\star^\bullet b$ in ${\mathbb Z}_p$ contains a term equal to $0$.

Examples: \  In ${\mathbb Z}_3$ we have that $\star(0) = 0,\,\star(1) = 2,\, \star(2) = 0$, and so $\star^2(b) = 0$ for every $b$. Thus $2\in{\mathbb P}_b$ for all $b$. On the other hand, in ${\mathbb Z}_5$, we have that $\star(0)=0,\,\star(1)=2,\,\star(2)=1,\star(3)=2,\,\star(4)=0$. So $\star^j\{1,2,3\}\subseteq\{1,2,3\}$ for every $j\ge0$. Hence $5\in{\mathbb P}_b$ if and only if either $b\equiv_50$ or $b\equiv_54$. 

In fact, infinitely many primes $p$ act like $5$ in the second example above, in that $p\in{\mathbb P}_b$ if and only if either $b\equiv_p0$ or $b\equiv_p -1$ for $p$. Working in ${\mathbb Z}_p$, $p >2$, we have that $\star x := x(x+1)=0$ if and only if either $x=0$ or $x=-1$. But $\star(x)=-1$ can happen only for $x\in{\mathbb Z}_p\setminus\{0,-1\}$. That $\star(x)=-1$ for some such $x$ is equivalent to the existence of a solution $x\in{\mathbb Z}_p\setminus\{0,-1\}$ of the quadratic equation $x^2+x+1=0$; that is, $\{(-1-\sqrt{-3})/2,(-1+\sqrt{-3})/2\}\cap({\mathbb Z}_p\setminus\{0,-1\})\not=\emptyset$. Thus there exists a solution $x$ if and only if $-3$ is a quadratic residue modulo $p$. 

Applying the Law of Quadratic Reciprocity, we see that 
\[
\bigg(\frac{-3}{p}\bigg)  = \bigg(\frac{-1}{p}\bigg)\bigg(\frac{3}{p}\bigg) = \bigg(\frac{-1}{p}\bigg)(-1)^{\frac{p-1}{2}}\bigg(\frac{p}{3}\bigg) = \bigg(\frac{p}{3}\bigg). 
\]
So we can infer for $p\not=3$ the following: $-3$ is a quadratic residue modulo $p$ \ $\Leftrightarrow$ \ $p$ is a quadratic residue modulo $3$ \ $\Leftrightarrow$ \ $p\equiv_3 1$.  So $\{b: p\in{\mathbb P}_b\} = \{b: b\equiv_p 0\vee b\equiv_p -1\}$ \ $\Leftrightarrow$ \ $p\equiv_3 2$, for $p\ge5$. These ideas lead to 

\begin{thm}\label{Can not get all primes} There is no $b\in{\mathbb N}$ for which ${\mathbb P}_b = {\mathbb P}$.   \end{thm} 

\begin{proof} Let $b$ be given. By Dirichlet's Theorem, there are infinitely many primes congruent to $2$ modulo $3$. Let $p>b+1$ be such a prime. Then $\neg(b\equiv_p 0)$ and $\neg(b\equiv_p -1)$. So, $p\not\in {\mathbb P}_b$ by the discussion above.  \end{proof}\

\section{Is $\sigma:{\cal E}(X)\rightarrow{\mathbb Q}^+$ bijective for some $X$?}

By Theorem \ref{Main1} we have that, for every $r \in {\mathbb Q}^+$, there are infinitely many $S \in {\cal E}$ such that $\sigma S=r$.  We call this phenomenon {\em hypersurjectivity}.

\begin{defn}
The function $f \colon A \rightarrow B$ is {\em hypersurjective} iff for all $b \in B$ the set of preimages of $b$ is infinite.
\end{defn}

Using this definition, we have that $\sigma:{\cal E}\rightarrow{\mathbb Q}^+$ is {\em hypersurjective}. This fact leads to the titular question of the present Section. The strongly negative answer  
below, provided by P\'{e}ter P. P\'{a}lfy, expanded our inquiry. \vspace{.3em}

Henceforth ${\cal F}$ denotes the family of all finite nonempty subsets of ${\mathbb Q}^+$, and ${\cal F}(X) := \{S:S \in {\cal F}\wedge S\subseteq X\}$ when $X\subseteq{\mathbb Q}^+$.  Notice 
that $\{\{1/j:j \in S\}$ for $S \in {\cal E}\}\subset{\cal F}$. We define the function $\Sigma:{\cal F}\rightarrow{\mathbb Q}^+$ by \[\Sigma: S\mapsto \Sigma S := \sum_{x \in S} x.\]

\begin{thm}\label{Palfy}{\bf (P\'{a}lfy)} \  There exists no $X\subseteq{\mathbb Q}^+$ for which $\Sigma:{\cal F}(X)\rightarrow{\mathbb Q}^+$ is bijective onto ${\mathbb Q}^+$.  \end{thm}

\begin{proof} Arguing by contradiction, we assume $\Sigma:{\cal F}(X)\rightarrow{\mathbb Q}^+$ is a bijection for a particular $X\subseteq{\mathbb Q}^+$. 

\begin{Lem} If $\{a,b\}\subseteq X$ with $a<b$, then $2a\le b$. \end{Lem}

\begin{proof} Assume that $a<b<2a$. Then $0<b-a<a$. So, since $\Sigma\restrict{\cal F}(X)$ is surjective onto ${\mathbb Q}^+$, there exists $Y \in {\cal F}(X)$ such that $\Sigma Y = b-a$. Note  
that $a\not\in Y$ since $b-a<a$. So $Y_1 := Y\cup\{a\} \in {\cal F}(X)$, and $\Sigma Y_1 = (b-a)+a = b = \Sigma\{b\}$. But $Y_1 \not= \{b\} \in {\cal F}(X)$, violating the assumption that 
$\Sigma\restrict{\cal F}(X)$ is injective. \end{proof}

To continue our proof of Theorem \ref{Palfy}, we now fix any $a \in X$. The lemma implies that we can list all of the elements in $X$ in an order-preserving way 
\[\cdots < x_{-3} < x_{-2} < x_{-1} < x_0 := a < x_1 < x_2 < x_3 < \cdots\]
and observe that $x_{-n} \le a/2^n$ and $x_n \ge 2^na$ for every $n \in {\mathbb N}$. Thus we have that the possibly infinite sum \[\sum_{n=1}^\infty x_{-n}\quad\mbox{satisfies the inequality}
\quad \sum_{n=1}^\infty x_{-n} \le \sum_{n=1}^\infty\frac{a}{2^n} = a.\]  If the strict inequality $x_{-n} < a/2^n$ holds for some $n \in {\mathbb N}$ then 
$a+\sum_{j=1}^\infty x_{-j} < 2a \le x_1.$ That is, then there can be no set $V \in {\cal F}(X)$ for which $\Sigma V$ is an element in the nonempty open interval 
$(a+\sum_{j=1}^\infty x_{-j},2a).$ Hence, by the surjectivity of $\Sigma\restrict{\cal E}(X)$, we see that $x_{-n} = a/2^n$ for every $n\in {\mathbb N}$. But $a \in X$ was chosen arbitrarily. 
Thus we have also that  $x_n=2^na$ for all $n \in {\mathbb N}$. That is to say, $X\subseteq \{2^ia:i \in {\mathbb Z}\}$. Of course this entails that there is no $q \in {\mathbb Q}^+$, the odd portion of 
whose denominator is coprime to the denominator of $a$, but such that $q$ is an element in the range of $\Sigma\restrict{\cal F}(X)$.  \end{proof}

P\'{a}lfy's theorem motivates a more general question: 

\begin{question}
If a set $X$ of positive rationals is such that $\Sigma\restrict{\cal F}(X)$ is a surjection onto ${\mathbb Q}^+$, then must $\Sigma\restrict{\cal F}(X)$ be hypersurjective? 
\end{question}

Attempts to argue by contradiction in order to provide an affirmative answer to this question resulted in several propositions. The  hypothesis for all of these propositions is that $\Sigma\restrict{\cal F}(X)$ is a surjective mapping onto ${\mathbb Q}^+$.\vspace{.5em}

\begin{defn} We call $g \in X\subseteq{\mathbb Q}^+$ {\em primitive for} $X$ iff $\{g\}$ is the only finite $S\subseteq X$ with $\Sigma S = g$.  \end{defn}

Slightly abusing language, we call the set $X\subseteq{\mathbb Q}^+$ itself surjective when the function $\Sigma:{\cal F}(X)\rightarrow{\mathbb Q}^+$ is a surjection, and we likewise call $X$ 
hypersurjective when the function $\Sigma\restrict{\cal F}(X)$ is hypersurjective onto ${\mathbb Q}^+$.\vspace{.5em} 

\begin{prop}\label{HasPrim}  A surjective $X\subseteq{\mathbb Q}^+$ has a primitive element if and only if $X$ is not hypersurjective.   \end{prop} 

\begin{proof} Suppose $X$ is surjective but not hypersurjective. For some $q \in {\mathbb Q}^+$, the family $\Sigma^-q := \{A: A \in {\cal F}(X) \wedge \Sigma A = q\}\not=\emptyset$ is finite. The union,  
$\bigcup\Sigma^-q\subset{\mathbb Q}^+$ of the pre-images of $q$ under $\Sigma$, is finite as well; hence $\bigcup\Sigma^-q$ contains a minimum $\overline{x}$. Choose $S \in \Sigma^-q$ with 
$\overline{x} \in S$. 

Suppose that $Y \in \Sigma^-\overline{x}$ is not $\{\overline{x}\}$. Now, $\overline{x}\not\in Y$ and $Y\cap S = \emptyset$, since every element in $S\setminus\{\overline{x}\}$ is larger than $\overline{x}>\max Y$. 
Let $Z := Y\cup(S\setminus\{\overline{x}\})$. Then $\Sigma Z=q$. This implies the absurdity, $Z \in \Sigma^-q$. But $\min Z < \overline{x} = \min\bigcup\Sigma^-q$. So there is no such $Y$, and thus 
$\Sigma^-\overline{x} = \{\{\overline{x}\}\}$. Hence  $\overline{x}$ is primitive for $X$. The converse is immediate.\end{proof}

\begin{prop}\label{Less} Let $z_1<z_2<\cdots < z_n < \overline{x}$ where the $z_i$ are elements in a surjective $X\subseteq{\mathbb Q}^+$, and where $\overline{x}$ is primitive for $X$. Then  
$\sum_{i=1}^n z_i < \overline{x}$.   \end{prop}

\begin{proof} Since $\overline{x}$ is primitive for $X$, we have  $\sum_{i=1}^n z_i \not=\overline{x}$. So pretend that $\sum_{i=1}^n z_i > \overline{x}$.  Then, for some $m$,   
\[ \sum_{i=m}^n z_i < \overline{x} < \sum_{i=m-1}^nz_i\quad\mbox{whence}\quad 0 < y := \overline{x}-\sum_{i=m}^n z_i < z_{m-1} < z_m < \cdots < z_n.\]  Since $X$ is surjective, we can 
choose $S \in {\cal F}(X)$ for which $\Sigma S = y$. If $s \in S$ then $s\le y$, and hence $s<z_i$ for every $i \in\{m,\ldots,n\}$. Therefore $S\cap\{z_m,\ldots, z_n\} = \emptyset$.  It follows that 
$\Sigma(S\cup\{z_m,\ldots, z_n\}) = \Sigma S + \sum_{i=m}^n z_i = y +\sum_{i=m}^n z_i = \overline{x}$ contrary to the hypothesis that $\overline{x}$ is primitive for $X$.     \end{proof}

\noindent{\bf Terminology.} When $\overline{x}$ is primitive for a surjective $X$, we write $\overline{x}^<$ to designate the set $X\cap(0,\overline{x})$.\vspace{.5em} 

The following assertion treats the sum of the infinite subset $\overline{x}^<\subseteq{\mathbb Q}^+$.

\begin{prop}\label{InfiniteSum} Let $\overline{x}$ be primitive for the subjective set $X\subseteq{\mathbb Q}^+$. Then $\displaystyle{\sum_{x \in \overline{x}^<} x = \overline{x}}$.     \end{prop} 

\begin{proof} Since  $\overline{x}^< := \{x_1,x_2,\ldots\}$ is denumerable, $\displaystyle{\sum_{x \in \overline{x}^<} x = \sum_{i=1}^\infty x_i}$. If either $\displaystyle{\sum_{i=1}^\infty x_i > \overline{x}}$ 
or $\displaystyle{\sum_{i=1}^\infty x_i = \infty,}$ then $\displaystyle{\sum_{i=1}^n x_i > \overline{x}}$ for some partial sum -- violating Lemma \ref{Less}. Hence, $\displaystyle{\sum_{i=1}^\infty x_i = y}$ for 
some $y\le\overline{x}$. But, if $y<\overline{x}$ then the interval $(y,\overline{x})$ contains elements $t$ with $\Sigma^-t = \emptyset$, contrary to the hypothesis that $X$ is surjective. \end{proof}

Observe that if $X$ is surjective then $0$ is an accumulation point of $X$. Furthermore, if $\overline{X}$ denotes the set of all primitive elements in $X$, then $0$ is the only accumulation point of $\overline{X}$ only if $\overline{X}$ contains no 
minimum element, and otherwise $\overline{X}$ has no accumulation points.\vspace{.5em}

Our efforts to prove that every surjective  $X$  is hypersurjective were stymied by an impediment appearing in sundry guises.  One of those guises is this: \  Given $q \in {\mathbb Q}^+$ and 
$\{A,B\}\subseteq{\cal F}(X)$  with $\Sigma A = \Sigma B = q $, is there a relevant procedure for replacing those two sets with a pairwise disjoint family ${\cal D}(A,B)\subseteq{\cal F}(X)$ for which 
$\Sigma C = \Sigma A$  for every  $C \in {\cal D}(A,B)$?  The answer is provided in \S6.
\vspace{1em}

The function $\Sigma$ induces two functions, $\nu^*:{\cal F}\rightarrow{\mathbb N}$ and $\delta^*:{\cal F}\rightarrow{\mathbb N}$, which are defined by 
\[\Sigma S = \frac{\nu^* S}{\delta^* S}\quad\mbox{where the integers}\quad \nu^* S\quad\mbox{and}\quad \delta^* S\quad\mbox{are coprime, for each}\quad S \in {\cal F}.\]\vspace{.5em}  

\begin{question}
What can be said about subfamilies ${\cal A\subseteq F}$ which are not of the form ${\cal F}(X)$ for some $X\subseteq{\mathbb Q}^+$? 
\end{question}

\begin{Exmp}  Let ${\cal X}$ be the family of all nonempty finite  sets $S$ of reduced fractions of the form $n/p^i$, where $p^i$ is the power of a prime $p$ for some integer $i\ge0$, 
where $n\ge1$, and where the set of denominators of the elements in $S$ is pairwise coprime. Of course ${\cal X} \subset {\cal F}$.  For every $n \in {\mathbb N}$, clearly $\delta^*(T) = n$ for an 
infinite family of $T \in {\cal X}$. However, ${\sf Range}(\Sigma\restrict{\cal X})\not={\mathbb Q}^+$; indeed,  $1/6 \not\in {\sf Range}(\Sigma\restrict{\cal X})$. Although the functions 
$\Sigma\restrict{\cal X}$ and $\nu^*\restrict{\cal X}$ are not injective, for all $q \in {\mathbb Q}^+$ we do have $|(\Sigma\restrict{\cal X})^-q|<\infty$ and $|\{S:S \in {\cal X}\wedge\nu^*S = q\}|<\infty$. 
\end{Exmp}

\section{Must every surjective set $X$ be hypersurjective?}

We will build a surjective set $X\subseteq{\mathbb Q}^+$ for which the only finite $S\subseteq X$ with $\Sigma S = 1$ is $S = \{1\}$, making $1$ primitive. This process begins with our 
revisiting the family ${\cal E}$ of all nonempty finite sets of positive integers. 

\begin{defn} We say that a set $K \in {\cal E}$ is a {\em binary basis} iff, for every positive integer $j\le\Sigma K$, there is some $A_j\subseteq K$ with $j=\Sigma A_j$. 
If $K$ is a binary basis then the integer $\Sigma K$ is said to be the {\em maximal sum} for $K$. More generally, an integer $k$ gets called a maximal sum iff $k$ is the 
maximal sum of some element in ${\cal E}$.  \end{defn}

The following provides insight into our reason for the expression ``binary basis'' in the definition above. We omit the easy elementary proof of 

\begin{prop}\label{Binary} The set of terms of each finite prefix of the sequence $\langle2^j\rangle_{j=0}^\infty$ is a binary basis. For each integer $n\ge0$, the maximal 
sum of the binary basis  $B_n := \{2^j: 0\le j \le n\}$ is $2^{n+1}-1$.     \end{prop}

We omit also a proof of 

\begin{lem}\label{Prefixes} If $\{a_1<a_2<\cdots<a_k\}$ is a binary basis, then  $\{a_1<\cdots<a_j\}$ is a binary basis when $j<k$. \end{lem}

\begin{lem}\label{MaxSums} There are arbitrarily long sequences of consecutive maximal sums of binary bases.  \end{lem}

\begin{proof} Pick $n \in {\mathbb N}$, and let $x$ be any integer with $2^n+1\le x\le 2^{n+1}$. We show that $K_x := \{1,2,\ldots, 2^n,x\}$ is a binary basis. Choose a  
positive integer $m\le\Sigma K_x$. 

Proposition \ref{Binary} notes that $B_n:=\{1,2,\ldots,2^n\}$ is a binary basis whose maximal sum is $2^{n+1}-1$. So if $m\le 2^{n+1}-1$ then $m = \Sigma S$ for some 
$S\subseteq B_n\subseteq K_x$. Thus we can take it that $2^{n+1}\le m \le \Sigma K_x$. So $m-x \le \Sigma K_x -x = \Sigma B_n$. Consequently there is a subset $S\subseteq B_n$ 
for which $\Sigma S = m-x$.  But then  $m =  (\Sigma S) +x = \Sigma (S\cup\{x\})$ while $S\cup\{x\}\subseteq K_x$.  

We have shown that $K_y$ is a binary basis for every integer $y$ with $2^n+1\le y \le 2^{n+1} = 2^n+2^n$. So $\langle 2^{n+1}+y\rangle_{y=0}^{2^n-1}$ is a sequence of $2^n$ 
consecutive integers each of which is a maximal sum of a binary basis.\end{proof}

 Lemma \ref{MaxSums} enables us to take leave of ${\cal E}$ and to return to our main present interest: The family ${\cal F}$ of nonempty finite sets of positive rational numbers. 
We use Lemma \ref{MaxSums} to construct recursively an infinite sequence $Y_2\subseteq Y_3\subseteq Y_4\subseteq\cdots$ of sets $Y_t$ of positive rationals and of  corresponding 
integers $m_t$ that will enable us to build the promised surjective but nonhypersurjective set $X\subseteq{\mathbb Q}^+$. This sequence $\langle\langle Y_t,m_t\rangle\rangle_{t=2}^\infty$ 
of pairs shall be designed to satisfy the following four criteria. For every integer $n\ge2$ we will contrive that: \vspace{.5em}
 
(1). \ $n|m_n$\vspace{.3em}

(2). \ $\displaystyle{ \min Y_n = \frac{1}{m_n}}$
\vspace{.3em}

(3). \ $\displaystyle{\Sigma Y_n = 1-\frac{1}{m_n}}$\vspace{.3em}

(4). \ $\displaystyle{ \left\{\frac{1}{m_n},\frac{2}{m_n},\ldots,\frac{m_n-1}{m_n}\right\}\subseteq \Sigma{\cal F}(Y_n)}\,\,$\footnote{The fact that we actually get set equality here is 
irrelevant to our argument ahead.}\vspace{.8em}

\noindent We induce on $n\ge2$. Define $Y_2 := \{1/2\}$ and $m_2 := 2$. Clearly the four criteria (1)-(4) hold for $\langle Y_2,m_2\rangle$. 

Suppose that $Y_2\subseteq Y_3\subseteq\cdots\subseteq Y_n$ and $m_2,m_3,\ldots,m_n$ have been specified, and that the four criteria hold for $\langle Y_t,m_t\rangle$ for every 
$t \in \{2,3,\ldots,n\}$. If $(n+1)|m_n$ then let $Y_{n+1} := Y_n$,  let $m_{n+1} := m_n$ and observe that the four criteria (1)-(4) hold for $\langle Y_{n+1},m_{n+1}\rangle$. So we will focus 
upon the situation where $m_n$ is not a multiple of $n+1$. It will become obvious that the situation $\neg\,( i+1|m_i)$ must occur for infinitely many $i$.\vspace{.5em} 

Note\footnote{Recall that when $I$ is a finite set of nonzero integers then $\mu I$ denotes the least common multiple of the elements in $I$.} that successive terms in the integer sequence 
$\displaystyle{\left\langle \frac{j\mu\{n+1,m_n\}}{m_n}-1\right\rangle_{j=1}^\infty }$ differ by the positive integer ${\displaystyle\frac{\mu\{n+1,m_n\}}{m_n}}$. Hence, by Lemma \ref{MaxSums}, 
there is a least positive integer $k$ for which $\displaystyle{\frac{k\mu\{n+1,m_n\}}{m_n}-1}$ is a maximal sum.  Let $K$ be a binary basis for which $\displaystyle{\Sigma K =  
\frac{k\mu\{n+1,m_n\}}{m_n} -1, }$ and define \[ m_{n+1} := k\mu\{n+1,m_n\}\qquad\mbox{and}\qquad S := \left\{\frac{x}{m_{n+1}}: x \in K\right\}\qquad\mbox{and then}\qquad Y_{n+1} := 
Y_n\cup S. \]

\noindent Clearly $Y_n\subseteq Y_{n+1}$. We must confirm that the pair $\langle Y_{m+1},m_{n+1}\rangle$ satisfies the criteria (1)-(4).\vspace{1em}

\textbf{(1) holds for $\langle Y_{n+1},m_{n+1}\rangle$:}
\  Since $n+1$ divides $\mu\{n+1,m_n\}$, it divides $k\mu\{n+1,m_n\} =: m_{n+1}$.\vspace{1em}

\textbf{(2) holds for $\langle Y_{n+1},m_{n+1}\rangle$:}
\  Of course $1\in K$. So $\displaystyle{\frac{1}{m_{n+1}} \in S\subseteq Y_{n+1}.}$ Plainly $\displaystyle{\frac{1}{m_{n+1}} = \min S}.$ By the inductive hypothesis, 
$\displaystyle{\min Y_n = \frac{1}{m_n} > \frac{1}{m_{n+1}}.}$ Thus $\displaystyle{\frac{1}{m_{n+1}} = \min Y_{n+1} }$. 
\vspace{1em}

\textbf{(3) holds for $\langle Y_{n+1},m_{n+1}\rangle$:}
\ By the inductive hypothesis, $\Sigma Y_n = 1-1/m_n$. 
We have also that 
\[ \Sigma S = \frac{\Sigma K}{m_{n+1}} =  \frac{\frac{k\mu\{n+1,m_n\}}{m_n}-1}{m_{n+1}} = \frac{\frac{m_{n+1}}{m_n}-1}{m_{n+1}} = \frac{1}{m_n} - \frac{1}{m_{n+1}} \]. 
Thus since $S = Y_{n+1}-Y_n$, we get by (2) that 
\[\Sigma(Y_{n+1}) = \Sigma(S)+\Sigma Y_n = \frac{1}{m_n}-\frac{1}{m_{n+1}}+1-\frac{1}{m_n} = 1-\frac{1}{m_{n+1}}.\]

\vspace{1em}

\textbf{(4) holds for $\langle Y_{n+1},m_{n+1}\rangle$:}
\ We know that $m_n|m_{n+1}$ since $m_{n+1} := k\mu\{n+1,m_n\}$. So $m_{n+1} = k'm_n$ for a positive integer $k'$. Now consider the 
positve rational number $d/m_{n+1}$, where $d < m_{n+1}$ is an integer. There are nonnegative integers $r<k'$ and $q$ for which $d = qk'+r$. Then 
\[ \frac{d}{m_{n+1}} = \frac{qk'+r}{m_{n+1}} = \frac{qk'}{m_{n+1}}+\frac{r}{m_{n+1}} = \frac{q}{m_n}+\frac{r}{m_{n+1}}. \] Now $0\le q < m_n$ since $0 < d/m_{n+1} <1$. By our 
definition of $S$, there is a subset $S_d \subseteq S$ for which \[ \Sigma S_d = \frac{r}{m_{n+1}}. \] If $q > 0$ then, by the inductive hypothesis, there is a subset $T_d\subseteq Y_n$ for 
which  \[ \Sigma T_d = \frac{q}{m_n}.\]  But then, since $S\cap Y_n = \emptyset$, we infer that 
\[ \Sigma(S_d\cup T_d) = \frac{q}{m_n}+\frac{r}{m_{n+1}} = \frac{qk'+r}{m_{n+1}} = \frac{d}{m_{n+1}}. \] This proves that 
\[\left\{\frac{1}{m_{n+1}}, \frac{2}{m_{n+1}},\ldots, \frac{m_{n+1}-1}{m_{n+1}} \right\} \subseteq \Sigma {\cal F}(Y_{n+1}).\] 

We have shown that each of the four criteria,  (1)-(4), is satisfied by the pair $\langle Y_{n+1},m_{n+1}\rangle$. \vspace{1em}

The induction just concluded establishes that there is an upward-nesting sequence $Y_2 \subseteq Y_3 \subseteq Y_4 \cdots $ of finite sets of positive rationals and an accompanying sequence 
$2 = m_2 \le m_3 \le m_4 \le \cdots$ of integers such that, for each integer $n\ge 2$, the pair $\langle Y_n,m_n\rangle$ satisfies the four criteria (1)-(4).\vspace{.5em} 

Our efforts culminate in our definition of this crucial set: \ \  $X := \bigcup_{n=2}^\infty Y_n\cup\{1,2,4,8,\ldots\} \subseteq{\mathbb Q}^+$.  

\begin{thm}\label{NotHyper}  The function $\Sigma:{\cal F}(X)\rightarrow{\mathbb Q}^+$ is a surjection which is not a hypersurjection. \end{thm}

\begin{proof} Note that $Y := \bigcup_{i=2}^\infty Y_i$ is disjoint from $\{2^j: j=0,1,2,\ldots\}$ since every element in $Y$ is less than $1$. 

We will show: \  [1] \ that ${\mathbb Q}^+ = \Sigma{\cal F}(X)$, and \  [2] \ that $S := \{1\}$ is the only $S \in {\cal F}(X)$ for which $\Sigma S = 1$.\vspace{.5em} 

[1]: \ Let $p \in {\mathbb Q}^+$, and write $p = m+r$ for $m\ge0$ an integer and $0\le r <1$.  Let $r>0$, and write $r = a/b$ where $\{a,b\}\subseteq{\mathbb N}$. Since the pair 
$\langle Y_b,m_b\rangle$ satisfies the criterion (1) above, we have that $b|m_b$, whence $m_b = cb$ for some $c \in {\mathbb N}$. It follows that $r = a/b = ac/m_b$. Note that $ac < m_b$ since 
$a < b$. Therefore, since $\langle Y_b,m_b\rangle$ satisfies the criterion (4) above, we have that there is a subset $V\subseteq Y_b\subseteq Y$ for which $r = ac/m_b = \Sigma V$.  

If $p = 0+r$ then $\Sigma V = p$,  But if $m>0$, then $m \in {\mathbb N} = \Sigma{\cal F}(\{1,2,4,8,\ldots\}) \subseteq\Sigma{\cal F}(X)$. So $m = \Sigma U$ for some finite $U\subseteq\{1,2,4,\ldots\}$. 
Moreover $V\cap U = \emptyset$ since $V$ contains no integers. So $V\cup U$ is a finite subset of $X$, and $\Sigma(V\cup U) = p$. We have established the surjectivity of 
$\Sigma\restrict{\cal F}(X)$.\vspace{1em} 

[2]:  \  Obviously $Y\subseteq(0,1)$. Pick an arbitrary finite subset $W\subseteq Y$. Then, since the set sequence $\langle Y_n\rangle_{n=1}^\infty$ nests upward, we have that $W \subseteq Y_\ell$ for some 
positive integer $\ell$.  By the criterion (3) of the four above, we have that $\Sigma W \le \Sigma V_\ell = 1-1/m_\ell < 1$.  So $\{1\}$ is the only finite subset $S\subseteq X$ for which $\Sigma S = 1$.       \end{proof}

\noindent{\large\bf Acknowledgments} \ Our efforts over a decade's span were aided by the counsel and support of friends and family. In particular, Arthur Tuminaro goaded one of us into the research from which our paper has grown. The paper's development benefited from conversations with Bob Cotton, Jacqueline Grace, and Rob Sulman. James Ruffo caught one of our early errors, allowing for its correction.

\vspace{1em} 

\noindent{\sf 2010 Mathematics Subject Classification}: \ 11A25, 
11N13

\end{document}